\documentclass[12pt]{amsart}
\usepackage{amsmath,amsthm,amssymb}
\usepackage{mathrsfs}
\usepackage{xcolor}
\font\teneufm=eufm10
\font\seveneufm=eufm7
\font\fiveeufm=eufm5
\newfam\frakturfam

\textfont\frakturfam=\teneufm
\scriptfont\frakturfam=\seveneufm
\scriptscriptfont\frakturfam=\fiveeufm

\theoremstyle{definition}
\newtheorem{definition}{Definition}[section]
\newtheorem{remark}[definition]{Remark}
\newtheorem{example}[definition]{Example}

\theoremstyle{plain}
\newtheorem{lemma}[definition]{Lemma}
\newtheorem{prp}[definition]{Proposition}
\newtheorem{corollary}[definition]{Corollary}
\newtheorem{theorem}[definition]{Theorem}

\begin{document}

\title{The symmetric operation in a free Novikov algebra}

\author{Askar Dzhumadil'daev}

\address{Institute of Mathematics and Mathematical Modeling, Almaty, 050010, Kazakhstan}

\email{askar56@gmail.com}

\author{Nurlan Ismailov}

\address{Astana IT University,
Mangilik El avenue, 55/11, Business center EXPO, block C1,
Astana, 010000, Kazakhstan and Institute of Mathematics and Mathematical Modeling, Almaty, 050010, Kazakhstan}

\email{nurlan.ismail@gmail.com}



\maketitle

\begin{abstract}
We study the symmetrization of the Novikov product. Using the embedding of a free Novikov algebra into a commutative differential algebra over a field of characteristic zero, we first obtain a basis for the subalgebra generated by the free generators with respect to the symmetrized product $a\circ b=ab+ba$. Next, using Euler operators (variational derivatives), we give a criterion for recognizing symmetric elements and show that these elements coincide with null Lagrangians in the corresponding differential realization. We then construct a non-special homomorphic image and show that the class of algebras embeddable into symmetrizations of Novikov algebras does not form a variety. Finally, we determine the module structure of the multilinear components over the symmetric group.
\end{abstract}

\section{\label{1}\ Introduction.}

A vector space $A$ over a field $\mathbb{F}$ equipped with the bilinear product $(x,y)\mapsto xy$ is called a {\it (right) Novikov algebra} if it satisfies the following two identities:
\begin{equation}\label{rsym=0}
(ab)c-a(bc)=(ac)b-a(cb),
\end{equation}
\begin{equation}\label{leftcom=0}
a(bc)=b(ac)
\end{equation}
for all $a,b,c\in A$.
Novikov algebras first appeared in the study of Hamiltonian operators in the formal calculus of variations by Gel'fand and Dorfman \cite{GelʹfandDorfman1979}, and later in the classification of linear Poisson brackets of hydrodynamic type by Balinskii and Novikov \cite{BalinskiiNovikov1985}.
Identity~(\ref{rsym=0}) defines the variety of right-symmetric (pre-Lie) algebras.

To provide examples of these algebras, consider the polynomial algebra
$P_n=\mathbb{F}[x_1,\ldots,x_n]$ over $\mathbb{F}$ in the variables $x_1,\ldots,x_n$ and the Witt algebra of index $n$,
that is, the Lie algebra of all derivations of $P_n$:
\[
W_n=\{\,u\partial_i \mid u\in P_n,\ 1\le i\le n\,\},
\]
where $\partial_i=\partial/\partial x_i$.
Define a product on $W_n$ by
\[
u\partial_i\cdot v\partial_j=(v\partial_j(u))\partial_i.
\]
Then $(W_n,\cdot)$ satisfies identity~(\ref{rsym=0}) and hence is a right-symmetric algebra (a right-symmetric Witt algebra).
Moreover, $(W_n,\cdot)$ is a Novikov algebra (a Novikov--Witt algebra) with respect to this product only in the case $n=1$.

The polynomial identities of right-symmetric Witt algebras were studied in
\cite{Dzhumadil’daev2000, Dzhumadil'daev-Lofwall2002, KozybaevUmirbaev2016}.
In particular, Kozybaev and Umirbaev \cite{KozybaevUmirbaev2016} proved that the family of right-symmetric algebras $W_n$,
$n\ge 1$, generates the variety of all right-symmetric algebras.
In \cite{Makar-LimanovUmirbaev2011} it was proved that, over a field of characteristic zero,
the variety of Novikov algebras is generated by the Novikov--Witt algebra $W_1$.

There is a general construction of Novikov algebras, the so-called {\it Gel'fand--Dorfman} construction, using commutative associative differential algebras \cite{GelʹfandDorfman1979}: if $A$ is a commutative associative algebra equipped with a derivation $D$, then the multiplication
\[
a\cdot b = D(a)b
\]
turns $A$ into a Novikov algebra.

In \cite{Dzhumadil'daev-Lofwall2002}, Dzhumadil'daev and L\"ofwall proved that every free Novikov algebra over a field of characteristic zero embeds into the free commutative associative differential algebra on the same generating set, and they also described a basis of the free Novikov algebra in terms of differential monomials. In \cite{DIU2023} using this embedding, Dotsenko, Ismailov, and Umirbaev proved that every set of identities of Novikov algebras over a field of characteristic zero follows from a finite set of identities.

Novikov algebras are {\it Lie-admissible}, that is, for every Novikov algebra $A$ the commutator
\[
[a,b]=a\cdot b-b\cdot a
\]
defines a Lie algebra structure on the underlying vector space of $A$.
It is clear that the Novikov--Witt algebra $(W_1,\cdot)$, equipped with the commutator bracket, is the Witt algebra $(W_1,[\,,\,])$.

Moreover, the Witt algebra $W_1$ satisfies the following identity of degree $5$:
\[
\sum_{\sigma\in S_4}(-1)^{\sigma}\,[x_{\sigma(1)},[x_{\sigma(2)},[x_{\sigma(3)},[x_{\sigma(4)},x_5]]]]=0,
\]
where $S_4$ is the symmetric group on $\{1,2,3,4\}$; see \cite{Razmyslov1986, Kirillov1989, Kirillov1991}.

In \cite{Molev1989}, Molev described the $S_n$-module structure of the space of multilinear Lie elements of degree $n$ in a free Novikov algebra. As corollaries, he obtained formulas for the multiplicities of irreducible $S_n$-modules and for its codimension. However, several long-standing open problems concerning the Witt algebra $W_1$ remain, including:
(i) finding a basis of the polynomial identities of $W_1$; (ii) constructing a linear basis of the space $W_1$; (iii) giving a criterion for determining which Lie polynomials belong to the Lie subspace of a free Novikov algebra.

In addition to the commutator on the Novikov--Witt algebra $(W_1,\cdot)$, one can also consider the anticommutator, i.e., the symmetrization of the Novikov product:
\[
a \circ b = a\cdot b + b\cdot a \qquad (a,b\in W_1).
\]
In this paper, we study this symmetrized product over a field of characteristic zero.

Let $Nov\langle X\rangle$ be the free Novikov algebra generated by a set $X$, and let $S\langle X\rangle$ be the smallest subspace of $Nov\langle X\rangle$ that contains $X$ and is closed under the symmetrized product. We say that $f\in Nov\langle X\rangle$ is {\it symmetric} if $f\in S\langle X\rangle$.

In \cite{Dzhumadil'daev2002, Dzhumadil'daev2005}, the first author of the present paper showed that, with respect to this symmetrized product, every Novikov algebra satisfies the following two identities:
\[
(a\circ b)\circ(c\circ d)-(a\circ d)\circ(c\circ b)
=(a,b,c)_{\circ}\circ d-(a,d,c)_{\circ}\circ b,
\]
and
\[
\begin{aligned}
&(((a\circ a)\circ a)\circ b)\circ b+(((a\circ b)\circ b)\circ a)\circ a
+2(((a\circ a)\circ b)\circ b)\circ a \\
&\quad +2(((a\circ b)\circ a)\circ a)\circ b
-3(((a\circ a)\circ b)\circ a)\circ b
-3(((a\circ b)\circ a)\circ b)\circ a=0.
\end{aligned}
\]
Here $(a,b,c)_{\circ}=a\circ(b\circ c)-(a\circ b)\circ c$ is the associator of~$\circ$.
Moreover, every identity of degree at most seven follows from these two identities.
The first identity is called the {\it Tortken identity}, and algebras satisfying it are called {\it Tortken algebras}.

Using the embedding theorem for free Novikov algebras into free commutative associative differential algebras \cite{Dzhumadil'daev-Lofwall2002}, we construct a basis of $S\langle X\rangle$ in terms of differential polynomials. As a corollary of this result, we obtain that the symmetrized product of any two Novikov polynomials is again a symmetric element. 

The characterization of symmetric elements in the free Novikov algebra $Nov\langle X\rangle$ is closely related to the {\it Euler operators} (or {\it variational derivatives}). In the calculus of variations, the Euler operator $E$ plays a fundamental role \cite{Olver1993}.  If $$I=\int L(x,y,y',y'',\ldots,)dx$$ is a variational problem, then the Euler (Euler--Lagrange) equation $E(L)=0$ forms a necessary condition for an extremal (maxima or minima) of functional $I$.  The integrand $L(x,y,y',y'',\ldots,)$ is called {\it Lagrangian} of the variatinal problem $I$.  A Lagrangian $L$ is said to be {\it null} if $E(L)=0$. We show that the space of null Lagrangians in a free Novikov algebra coincides with the space of symmetric elements. The proof of this criterion crucially uses the generalized Gel'fand--Dikii construction introduced in \cite{OlverShakiban1978, Shakiban1981}. 

Then, using Cohn's criterion \cite[Theorem 2.2]{Cohn1954}, we show that the class of algebras embeddable into Novikov algebras with respect to the symmetrized product does not form a variety.

The $S_n$-module structure of the space of multilinear elements in a free Novikov algebra was completely described in \cite{Dzhumadil'daevIsmailov2014} by the authors of the present paper. Using these methods, we prove that the subspace of multilinear elements in $S\langle X\rangle$ as $S_n$-module is isomorphic to a direct sum of modules induced from the trivial modules of certain Young subgroups, and that the multiplicities of Specht modules in this decomposition are given by Kostka numbers.

This paper is organized as follows. In Section~2, we fix the notation and recall the basis of the algebra of differential polynomials with one derivation, as well as the description of a basis of free Novikov algebras in terms of differential monomials. In Section~3, we construct a basis for the subalgebra generated by the free generators with respect to the symmetrized Novikov product. Section~4 recalls the generalized Gel'fand--Dikii transform and establishes a connection between null Lagrangians and symmetric elements in the free Novikov algebra. In Section~5, using Cohn's criterion, we show that the class of algebras embeddable into symmetrizations of Novikov algebras does not form a variety. Finally, Section~6 describes the $S_n$-module structure of the space of null Lagrangians in the free Novikov algebra.

\section{\label{2}\ A differential polynomial algebra and free Novikov algebra}

Let $\mathbb{F}$ be a field of characteristic zero, and let $A$ be an $\mathbb{F}$-algebra. 
A linear map $D\!:A\to A$ is called a \emph{derivation} of $A$ if it satisfies the Leibniz rule
$$D(ab)=D(a)\,b+a\,D(b)$$
for all $a,b\in A$.

Let $X$ be the set of variables $x_1,x_2,\ldots,x_n$, and let $X^{D}$ denote the set of variables of the form $D^{k}(x_i)$ with $1\le i\le n$ and $k\ge 0$. For brevity, we write $u'$, $u''$, and $u^{(k)}$ for $D(u)$, $D^{2}(u)$, and $D^{k}(u)$, respectively. Denote by $\mathbb{F}\{X\}$ the differential (polynomial) algebra over $\mathbb{F}$ generated by $X^{D}$ with a single derivation $D$. This is the \emph{ordinary differential algebra}; see, for example, \cite{DusengaliyevaUmirbaev2018,Kolchin1973}.

Now we recall a monomial basis of the algebra $\mathbb{F}\{X\}$. 
Let $I=(i_1,\ldots,i_{m})$ be a (finite) multi-index with $i_r\in\mathbb{Z}_{\ge 0}$. 
For $x\in X$, define the differential monomial
$$
x^{I}=x^{(i_1)}\cdots x^{(i_{m})}.
$$
We call $m$ the \emph{(standard) degree} of $x^I$ and write $\deg(x^I)=m$, and we call
$$
d(x^I)=i_1+\cdots+i_m
$$
the \emph{differential degree} of $x^I$.

Consider monomials of the form
$$
w=x^{I_1}_1\cdots x^{I_n}_n,
$$
where each $I_k$ is a multi-index as above. The set
$$
\mathcal{M}=\{\,x^{I_1}_1\cdots x^{I_n}_n \mid I_1,\ldots,I_n \text{ multi-indices}\,\}
$$
forms an $\mathbb{F}$-basis of $\mathbb{F}\{X\}$.

For $w=x^{I_1}_1\cdots x^{I_n}_n$, define
$$
\deg_{x_i}(w)=\deg\!\bigl(x_i^{I_i}\bigr)=|I_i|,
$$
i.e., the number of factors $x_i^{(k)}$ occurring in the $i$-th block (counted with multiplicity).

We extend the two degrees additively:
$$
\deg(w)=\deg(x^{I_1}_1)+\cdots+\deg(x^{I_n}_n),\qquad
d(w)=d(x^{I_1}_1)+\cdots+d(x^{I_n}_n).
$$
In particular,
$$
\deg(w)=\sum_{i=1}^n \deg_{x_i}(w).
$$

Every $f\in\mathbb{F}\{X\}$ can be written as $f=\sum_{\alpha} c_{\alpha} w_{\alpha}$ with $c_{\alpha}\in\mathbb{F}$ and $w_{\alpha}\in\mathcal{M}$. 
We say that $f$ is \emph{homogeneous} if there exists $(\lambda_1,\ldots,\lambda_n)\in\mathbb{Z}_{\ge0}^n$ such that
$$
\deg_{x_i}(w_{\alpha})=\lambda_i\quad\text{for every }i=1,\ldots,n\text{ and all }\alpha.
$$
In this case we set $\deg_{x_i}(f)=\lambda_i$ and call $(\lambda_1,\ldots,\lambda_n)$ the \emph{multidegree} of $f$.

As mentioned earlier, a commutative associative $\mathbb{F}$-algebra $A$ with a derivation $D$ becomes a Novikov algebra $(A,\cdot)$ by setting
$a\cdot b:=D(a)\,b$ for all $a,b\in A$.
Therefore, $\mathbb{F}\{X\}$ is a Novikov algebra under this product.  
Let $Nov\langle X\rangle$ denote the subalgebra of $(\mathbb{F}\{X\},\cdot)$ generated by $X$. Then $Nov\langle X\rangle$ is the free Novikov algebra on $X$ without a unit \cite{Dzhumadil'daev-Lofwall2002}.  
There are two known bases for free Novikov algebras: one described via Young diagrams \cite{Dzhumadil'daev2011}, and one given by differential monomials \cite{Dzhumadil'daev-Lofwall2002,DusengaliyevaUmirbaev2018}. In this paper, we work with  the second basis.

\begin{theorem}\label{Th1}
The set
$$
\mathcal{N}=\{\,u\in\mathcal{M}\mid \deg(u)-d(u)=1\,\}
$$
is a basis of the free Novikov algebra $Nov\langle X\rangle$.
\end{theorem}

Let $F_m$ denote the homogeneous component of degree $m$ of $Nov\langle X\rangle$, and let $F^{\mathrm{multi}}_n$ be its multilinear component in $x_1,\ldots,x_n$. 
It is shown in \cite{Dzhumadil'daev2011} that
$$
\dim F^{\mathrm{multi}}_n=\binom{2n-2}{\,n-1\,}.
$$

\section{\label{2}\ A basis for the space of symmetric elements}

Before presenting a basis for the space of symmetric elements, we record the following identity, valid in any differential algebra over a field of characteristic $\neq 2$.

\begin{lemma}\label{lemma1} Every differential algebra over a field of characteristic different from two satisfies the following identity
$$2(a'bc)'=((ab)'c)'+((ac)'b)'-((bc)'a)'.$$
\end{lemma}

The next statement provides a linear basis for the space of symmetric elements in the free Novikov algebra $Nov\langle X\rangle$, expressed in terms of differential polynomials.

\begin{theorem}\label{basis}
The set $\overline{\mathcal{M}}:=X\cup\{\,u'\mid u\in\mathcal{M},\ \deg(u)-d(u)=2\,\}$ is a basis of the space of symmetric elements $S\langle X\rangle$.
\end{theorem}

\begin{proof} 
Let $\operatorname{span}_{\mathbb{F}}\overline{\mathcal{M}}$ denote the linear span of $\overline{\mathcal{M}}$. 
We first show that $S\langle X\rangle \subseteq \operatorname{span}_{\mathbb{F}}\overline{\mathcal{M}}$. Since $S\langle X\rangle$ is the smallest subspace containing $X$ and closed under the symmetric product $a\circ b:=(ab)'$, it suffices to verify:
(i) $x_i\circ x_j\in  \overline{\mathcal{M}}$, for $x_i,x_j\in X$; 
(ii) $w\circ x_i\in \operatorname{span}_{\mathbb{F}}\overline{\mathcal{M}}$ for $w\in \overline{\mathcal{M}}$, $x_i\in X$; 
(iii) $w\circ v\in \operatorname{span}_{\mathbb{F}}\overline{\mathcal{M}}$ for $w,v\in \overline{\mathcal{M}}$.

\smallskip
\emph{(i)} For $x_i,x_j\in X$ we have 
$x_i\circ x_j=(x_ix_j)' \in \overline{\mathcal{M}}$, since $u=x_ix_j$ satisfies $\deg(u)-d(u)=2$.

\smallskip
\emph{(ii)} Let $w\in \overline{\mathcal{M}}$ with $\deg(w)>1$. Then $w=u'$ for some $u=x_{l_1}^{(k_1)}\cdots x_{l_n}^{(k_n)}\in\mathcal{M}$ such that $\deg(u)-d(u)=2$. For $x_i\in X$,
$$
w\circ x_i=(u' x_i)'=\Bigl(\sum_{j=1}^{n} x_{l_1}^{(k_1)}\cdots x_{l_j}^{(k_j+1)}\cdots x_{l_n}^{(k_n)} x_i\Bigr)' 
= \sum_{j=1}^{n} \bigl(u_j\,x_i\bigr)',
$$
where $u_j:=x_{l_1}^{(k_1)}\cdots x_{l_j}^{(k_j+1)}\cdots x_{l_n}^{(k_n)}$. 
For each $j$, $\deg(u_j)=\deg(u)$ and $d(u_j)=d(u)+1$, hence
$$
\deg(u_jx_i)-d(u_jx_i)=(\deg(u)+1)-(d(u)+1)=\deg(u)-d(u)=2,
$$
so $(u_jx_i)'\in \overline{\mathcal{M}}$. Therefore $w\circ x_i\in \operatorname{span}_{\mathbb{F}}\overline{\mathcal{M}}$.

\smallskip
\emph{(iii)} Let $v=z'$ with $z=x_{s_1}^{(r_1)}\cdots x_{s_m}^{(r_m)}\in\mathcal{M}$ and $\deg(z)-d(z)=2$. Then
\begin{align*}
w\circ v&=(u' z')' 
=\Bigl(\sum_{j=1}^{n} u_j \cdot \sum_{q=1}^{m} z_q\Bigr)' 
= \sum_{j=1}^{n}\sum_{q=1}^{m} (u_j z_q)',
\end{align*}
where $w_j$ and $z_q$ are obtained from $w$ and $z$ by increasing the $j$-th and $q$-th differential orders by $1$, respectively. Consequently
$$\deg(u_j z_q)-d(u_j z_q)$$
$$=(\deg(u)+\deg(z))-(d(u)+1+d(z)+1)
=(\deg(u)-d(u))+(\deg(z)-d(z))-2=2.$$

Hence each $(u_j z_q)'\in \overline{\mathcal{M}}$, and so $w\circ v\in \operatorname{span}_{\mathbb{F}}\overline{\mathcal{M}}$. Combining (i)–(iii), we conclude $S\langle X\rangle \subseteq \operatorname{span}_{\mathbb{F}}\overline{\mathcal{M}}$.

Now we prove that every element of $\overline{\mathcal{M}}$ belongs to $S\langle X\rangle$. 
We proceed by induction on the degree $\deg(w)$ of elements $w\in\overline{\mathcal{M}}$.

\smallskip
\emph{Base cases.}
If $\deg(w)=1$, then $w=x_i\in S\langle X\rangle$.
If $\deg(w)=2$, then necessarily $w=(x_ix_j)'=x_i\circ x_j\in S\langle X\rangle$.

\smallskip
\emph{Inductive step.}
Assume the claim holds for all elements of $\overline{\mathcal{M}}$ of degree $<m$.
Let $w=u' \in \overline{M}$ with $\deg(u)=m>2$ and $\deg(u)-d(u)=2$, where
\[
u=x_{l_1}^{(k_1)}\cdots x_{l_m}^{(k_m)}\in\mathcal{M}.
\]
To prove $u'\in S\langle X\rangle$ we consider two cases.

\smallskip
\textbf{Case 1: $u$ contains a first–order factor.}
Suppose some $x_i'$ occurs in $u$.  Since $\deg(u)-d(u)=2$, we can choose $x_j\in X$ and a monomial $u_0$ such that
\[
u=x_i' x_j u_0,\qquad \deg(u_0)-d(u_0)=1.
\]
By Lemma~\ref{lemma1} we have
\[
u'=(x_i' x_j u_0)'=\tfrac12\Big(((x_i x_j)'u_0)'+((x_i u_0)'x_j)'-((x_j u_0)'x_i)'\Big).
\]
Here $\deg(x_iu_0)=\deg(u_0)+1<m$ and $\deg(x_iu_0)-d(x_iu_0)=2$, hence 
\((x_iu_0)'\), \((x_ju_0)'\in S\langle X\rangle\) by the induction hypothesis; therefore
\(((x_i u_0)'x_j)'=(x_iu_0)'\circ x_j\) and \(((x_ju_0)'x_i)'=(x_ju_0)'\circ x_i\) lie in $S\langle X\rangle$ as well.

It remains to consider the term $((x_i x_j)'u_0)'$. Since $(x_i x_j)'=x_i\circ x_j\in S\langle X\rangle$, set $t=x_i\circ x_j$. Introducing a new variable $y$, we have $\deg(yu_0)=m-1$ and $d(yu_0)=m-3$. Hence, by the induction hypothesis, $(yu_0)'\in S\langle X,y\rangle$. Substituting $t$ for $y$, we obtain $(tu_0)'\in S\langle X\rangle$, that is,
$((x_i x_j)'u_0)'\in S\langle X\rangle$. Thus all three summands are in $S\langle X\rangle$, and hence $u'\in S\langle X\rangle$.

\smallskip
\textbf{Case 2: $u$ contains no first–order factor.} Since $\deg(u)-d(u)=2$ and no factor $x_i'$ occurs, there exist differential monomials $u_1$ and $u_2$ such that $u=u_1u_2$, where $u_1=x_i^{(k)}x_{j_1}\cdots x_{j_k}$ with $k>1$  for some $i$.  Then 
\[
u'=(x_i^{(k)}x_{j_1}\cdots x_{j_k}u_2)'=\big((x_i^{(k-1)}x_{j_1}\cdots x_{j_k})'u_2\big)'
-\sum_{l=1}^{k}\big(x_i^{(k-1)}x_{j_1}\cdots x_{j_l}'\cdots x_{j_k}u_2\big)'.
\]

For the first term, we have
$\deg(u_2)-d(u_2)=\deg(u)-(k+1)-\bigl(d(u)-k\bigr)=1$.
By the induction hypothesis,
$t=\left(x_i^{(k-1)}x_{j_1}\cdots x_{j_k}\right)'\in S\langle X\rangle$.
Introducing a new variable $y$, we have
$\deg(yu_2)-d(yu_2)=2$ and $\deg(yu_2)<\deg(u)$.
Hence, by the induction hypothesis,
$(yu_2)'\in S\langle X,y\rangle$.
Substituting $t$ for $y$, we obtain
$(tu_2)'\in S\langle X\rangle$.

Each summand in the remaining finite sum contains a first-order factor
$x_{j_l}'$ and therefore belongs to $S\langle X\rangle$ by Case~1.

This completes the induction and hence every element of $\overline{\mathcal{M}}$ belongs to $S\langle X\rangle$.

Suppose that $\left(\sum_u \lambda_u u\right)'=0$.
The derivation $D$ has trivial kernel on the subspace of differential polynomials of positive degree. Therefore,
$\sum_u \lambda_u u=0$.
Since the differential monomials form a basis of the polynomial differential algebra, it follows that $\lambda_u=0$ for all $u$. Thus $\overline{\mathcal{M}}$ is linearly independent. Consequently, $\overline{\mathcal{M}}$ is a basis of $S\langle X\rangle$.

\end{proof}

\begin{remark}
The proof of Theorem~\ref{basis} yields a recursive algorithm for expressing every symmetric differential monomial $u'$, where $\deg(u)-d(u)=2$, as a linear combination of nonassociative monomials with respect to the product $\circ$.

If $\deg(u)=2$, then $u=x_ix_j$, and therefore $u'=x_i\circ x_j$.

Assume now that $\deg(u)>2$.

\textbf{Case 1.} Suppose that $u$ contains a first-order factor. Then one can write $u=x_i'x_ju_0$, where $\deg(u_0)-d(u_0)=1$. By Lemma~\ref{lemma1},
$u'=\frac12\left(((x_ix_j)'u_0)'+((x_iu_0)'x_j)'-((x_ju_0)'x_i)'\right)$.
The elements $(x_iu_0)'$ and $(x_ju_0)'$ have smaller degree and are therefore expressed recursively as $\circ$-polynomials. Hence the second and third terms are obtained by multiplying these expressions by $x_j$ and $x_i$, respectively, with respect to $\circ$.

For the first term, introduce a new variable $y$. Since $\deg(yu_0)-d(yu_0)=2$ and $\deg(yu_0)<\deg(u)$, the element $(yu_0)'$ is expressed recursively as a $\circ$-polynomial in $X\cup{y}$. Substituting $x_i\circ x_j$ for $y$ gives the required expression for $((x_ix_j)'u_0)'$.

\textbf{Case 2.} Suppose that $u$ contains no first-order factor. Since $\deg(u)>2$ and $\deg(u)-d(u)=2$, the monomial $u$ contains a factor $x_i^{(k)}$ with $k>1$. Choose $k$ undifferentiated factors $x_{j_1},\ldots,x_{j_k}$ and write $u=x_i^{(k)}x_{j_1}\cdots x_{j_k}u_2$. Then
$u'=\left(\left(x_i^{(k-1)}x_{j_1}\cdots x_{j_k}\right)'u_2\right)'
-\sum_{l=1}^{k}\left(x_i^{(k-1)}x_{j_1}\cdots x_{j_l}'\cdots x_{j_k}u_2\right)'$.

The element $\left(x_i^{(k-1)}x_{j_1}\cdots x_{j_k}\right)'$ has smaller degree and is first expressed recursively as a $\circ$-polynomial. Introducing a new variable $y$, one then expresses $(yu_2)'$ recursively and substitutes the obtained $\circ$-polynomial for $y$. Each summand in the remaining finite sum contains a first-order factor and is therefore reduced by Case~1.

At each step, the procedure reduces the problem to differential monomials of smaller degree. Hence, after finitely many steps, one reaches the initial case $\deg(u)=2$. Therefore every symmetric differential monomial $u'$ can be expressed effectively as a linear combination of nonassociative $\circ$-monomials in the generators $X$.
\end{remark}
\begin{example}
Consider the symmetric differential monomial $(x''x^3)'$. Since $\deg(x''x^3)-d(x''x^3)=4-2=2$ and $x''x^3$ contains no first-order factor, we first apply Case~2. Taking $k=2$, we obtain
$(x''x^3)'=((x'x^2)'x)'-2((x')^2x^2)'$.

The first term is reduced by Case~1. Indeed,
$(x'x^2)'=\frac12((x^2)'x)'=\frac12(x\circ x)\circ x$,
and hence
$((x'x^2)'x)'=\frac12((x\circ x)\circ x)\circ x$.

It remains to express $((x')^2x^2)'$. Applying Case~1 to the monomial $(x')^2x^2$, we obtain
$((x')^2x^2)'=\frac12(((x^2)'x'x)')$.
To treat the latter term, introduce a new variable $y$. By Case~1,
$(yx'x)'=\frac12(x\circ x)\circ y$.
Substituting $y=x\circ x=(x^2)'$, we get
$(((x^2)'x'x)')=\frac12(x\circ x)\circ(x\circ x)$.
Therefore,
$((x')^2x^2)'=\frac14(x\circ x)\circ(x\circ x)$.

Combining the above expressions, we finally obtain
$(x''x^3)'=\frac12((x\circ x)\circ x)\circ x-\frac12(x\circ x)\circ(x\circ x)$.

Thus,
$(x''x^3)'=\frac12\left(((x\circ x)\circ x)\circ x-(x\circ x)\circ(x\circ x)\right)$.
This example illustrates the recursive procedure in which Case~2 is applied first and the resulting terms are subsequently reduced using Case~1.
\end{example}

As an immediate consequence of Theorem \ref{basis}, we obtain the following structural property of the algebra $S\langle X\rangle$.

\begin{corollary}\label{ideal}
We have
\[Nov\langle X\rangle\circ Nov\langle X\rangle
\subseteq S\langle X\rangle.
\]
In particular, $S\langle X\rangle$ is an ideal of the symmetrized algebra
$\operatorname{Nov}\langle X\rangle^{(+)}$.
\end{corollary}
\begin{proof}
Let $f,g$ be basis elements of $Nov\langle X\rangle$. Then $\deg(f)-d(f)=\deg(g)-d(g)=1$, and hence $\deg(fg)-d(fg)=2$. Since $f\circ g=(fg)'$, by Theorem \ref{basis}, we have $f\circ g\in S\langle X\rangle$. Therefore,
$Nov\langle X\rangle\circ Nov\langle X\rangle\subseteq S\langle X\rangle$.
\end{proof}

Write $Sym_n$ for the subspace of symmetric elements in $F^{\mathrm{multi}}_n$.  We now compute the dimension of the symmetric subspace.

\begin{corollary}\label{dimension}
The dimension of the space $Sym_n$ of symmetric elements in $F^{\mathrm{multi}}_n$ is given by $
\dim Sym_1=1,
$
and, for $n>1$,
$$
\dim Sym_n
=
\frac{1}{2}\dim F^{\mathrm{multi}}_n
=
\frac{1}{2}\binom{2n-2}{n-1}.
$$
\end{corollary}
\begin{proof}
By Theorem~\ref{basis}  the description of symmetric elements, a basis of $Sym_n$ consists of differential polynomials of the form
$$
\bigl(x^{(i_1)}_{1}\cdots x^{(i_n)}_{n}\bigr)' \qquad\text{with}\qquad
\deg(x^{(i_1)}_{1}\cdots x^{(i_n)}_{n})-d\bigl(x^{(i_1)}_{1}\cdots x^{(i_n)}_{n}\bigr)=2.
$$
In the multilinear component we have $\deg(x^{(i_1)}_{1}\cdots x^{(i_n)}_{n})=n$, so the condition is
$$
i_1+\cdots+i_n = n-2.
$$
Thus $\dim Sym_n$ equals the number of weak compositions of $n-2$ into $n$ parts, which by stars–and–bars is
$$
\binom{(n-2)+n-1}{\,n-1\,}=\binom{2n-3}{\,n-1\,}.
$$
Using the elementary identity
$$
\binom{2n-3}{\,n-1\,}=\frac{1}{2}\binom{2n-2}{\,n-1\,},
$$
we obtain the stated formula.
\end{proof}

\section{\label{2}\ A criterion for symmetric elements}
In this section, we present a criterion for determining symmetric elements in free Novikov algebras $Nov\langle X\rangle$.  The criterion will be given in terms of the Euler operators.  

By definition (see for example \cite{Olver1993}) {\it $k$-th Euler operator or the variational derivative} with respect to the variable $x_k$ is defined by 
$$E^{k}=\sum_{i=0}^{\infty}(-D)^i\frac{\partial}{\partial x^{(i)}_k}.$$
The (total) {\it Euler operator}  
$$E:\mathbb{F}\{X\}\rightarrow \mathbb{F}\{X\}^{n}=\underbrace{\mathbb{F}\{X\}\times \cdots\times \mathbb{F}\{X\}}_n$$
is defined by $E(f)=(E^1(f),\ldots, E^n(f))$ for $f\in \mathbb{F}\{X\}$. 

Gel'fand and Dikii \cite{GelʹfandDikii1975} studied the exactness of a chain complex involving the operators $E$ and $D$ over a field of characteristic zero, in the case where the set $X$ contains a single variable $x_1$. Later, Olver and Shakiban \cite{OlverShakiban1978}, building on the Gel'fand-Dikii transform, extended this complex by introducing additional operators.
For the present discussion, we record only the following subsequence of the operators they considered. Specifically,

\begin{theorem}[\cite{GelʹfandDikii1975}, \cite{OlverShakiban1978}]\label{GelʹfandDikii1975}
The complex
$$0\xrightarrow[]{}\mathbb{F}\{x_1\}\xrightarrow[]{D}\mathbb{F}\{x_1\}\xrightarrow[]{E^1}\mathbb{F}\{x_1\}$$
is exact. In other words, $E^1(f)=0$ if and only if there exists $g\in\mathbb{F}\{x_1\}$ such that $f=D(g)$.
\end{theorem}

Subsequently, Shakiban \cite{Shakiban1981} generalized this result in two directions: first, to the case $n\ge 1$; and second, to the setting where the variables $x_1,\ldots,x_n$ depend on independent variables $t_1,\ldots,t_m$, namely $x_i=x_i(t_1,\ldots,t_m)$. For our purposes, we take $m=1$, i.e., $x_i=x_i(t_1)$, with $D$ defined by $D_1=\tfrac{d}{dt_1}$. In this specialization, her theorem takes the following form:

\begin{theorem}[\cite{Shakiban1981}]\label{Shakiban1981}
The complex
$$0\xrightarrow[]{}\mathbb{F}\{X\}\xrightarrow[]{D}\mathbb{F}\{X\}\xrightarrow[]{E}\mathbb{F}\{X\}$$
is exact. In other words, $E^k(f)=0$ for all $k\in\{1,\ldots,n\}$ if and only if there exists $g\in\mathbb{F}\{X\}$ such that $f=D(g)$.
\end{theorem}

Under the assumptions of Theorem \ref{Shakiban1981}, we will show that if $f=f(x_1,\ldots,x_n)\in Nov(X)$ is homogeneous and $\deg_{x_i}(f)>0$ for every $i$, then the vanishing of $E^k(f)$ for some $k\in\{1,\ldots,n\}$ forces $E^j(f)=0$ for all $j\in\{1,\ldots,n\}$. Equivalently, for such $f$, the conditions $E^1(f)=\cdots=E^n(f)=0$ are all equivalent.

First, we briefly recall the \emph{generalized Gel'fand–Dikii transform}, introduced in \cite{Shakiban1981}. We present it in the case relevant to us, namely $x_i = x_i(t_1)$ for all $i$, over a field $\mathbb{F}$ of characteristic zero.

Define the symmetric algebra of $\mathbb{F}^k$
$$S^k=\bigodot\mathbb{F}^k=\bigoplus_{i=0}^{\infty}\bigodot_i\mathbb{F}^{k},$$
and let $\mathscr{L}_N=S^{N_1}\otimes\cdots \otimes S^{N_n}$. Let $\{e_{k1},\ldots,e_{kN_k}\}$ be a basis of $\mathbb{F}^{N_k}$. For a multi-index $I=(i_1,\ldots,i_{N_k})\in\mathbb{Z}_{\ge0}^{N_k}$ set
$$e_k^{I}:=e_{k1}^{\,i_1}\odot \cdots \odot e_{kN_k}^{\,i_{N_k}},$$
so that the elements $e_k^{I}$ (as $I$ varies) form a basis of $S^{N_k}$.

Then the simple tensors
$$e_1^{I_1}\otimes \cdots \otimes e_n^{I_n}$$
form a basis of $\mathscr{L}_N$. Thus, if $\phi\in\mathscr{L}_N$, we may regard $\phi$ as a polynomial in the variables $e_{kj}$ ($1\le k\le n$, $1\le j\le N_k$) and write
$$\phi=\phi(E_1,\ldots,E_n),\qquad E_k:=(e_{k1},\ldots,e_{kN_k}).$$
Let $\pi_k\in S_{N_k}$ be a permutation of $\{1,\ldots,N_k\}$ acting on $E_k$ by
$$\pi_k E_k=(e_{k\pi_k(1)},\ldots,e_{k\pi_k(N_k)}).$$
For $\pi=(\pi_1,\ldots,\pi_n)$ with $\pi_k\in S_{N_k}$, define the induced action on $\mathscr{L}_N$ by
$$\widehat{\pi}(\phi)=\phi(\pi_1E_1,\ldots,\pi_nE_n).$$

Now define the symmetrization map
$$\sigma=\frac{1}{N_1!\cdots N_n!}\sum_{\pi}\,\widehat{\pi}:\ \mathscr{L}_N\longrightarrow \mathscr{L}_N,$$
where the sum ranges over all $n$-tuples $\pi=(\pi_1,\ldots,\pi_n)$ with $\pi_k\in S_{N_k}$. Set $\widehat{\mathscr{L}_N}=\sigma(\mathscr{L}_N)$, and write
$$\mathscr{L}=\bigoplus_{N}\mathscr{L}_N,\qquad \widehat{\mathscr{L}}=\bigoplus_{N}\widehat{\mathscr{L}_N}.$$

By \cite[Def.~3]{Shakiban1981}, the generalized Gel'fand–Dikii transform is the linear map
\[
\mathscr{F}:\mathbb{F}\{X\}\to \mathscr{L}
\]
defined on monomials by
$$
\mathscr{F}\bigl(x_1^{I_1}\cdots x_n^{I_n}\bigr)
=\sigma\bigl(e_1^{I_1}\otimes \cdots \otimes e_n^{I_n}\bigr)
=\frac{1}{N_1!\cdots N_n!}
\sum_{\pi_1\in S_{N_1},\,\ldots,\,\pi_n\in S_{N_n}}
e_1^{\pi_1 I_1}\otimes \cdots \otimes e_n^{\pi_n I_n},
$$
where each $I_k=(i_{k1},\ldots,i_{kN_k})$ is a multi-index and
\[
\pi_k I_k = \bigl(i_{k\,\pi_k(1)},\ldots,i_{k\,\pi_k(N_k)}\bigr)
\quad\text{for } \pi_k\in S_{N_k},\ k=1,\ldots,n.
\]

\begin{theorem}[\cite{Shakiban1981}, Thm.~4]\label{thm:Shakiban81-iso}
The generalized Gel'fand–Dikii transform is a linear isomorphism
\[
\mathscr{F}:\mathbb{F}\{X\}_N \longrightarrow \widehat{\mathscr{L}_N}.
\]
\end{theorem}

If $f\in \mathbb{F}\{X\}$ and $\mathscr{F}(f)=\phi$, the notation $\widehat{E^k}(\phi)$ will denote $\mathscr{F}\!\bigl(E^k(f)\bigr)$.

\begin{lemma}[\cite{Shakiban1981}, Lemma~5(iii)]\label{lem:Shakiban}
For $\phi \in \mathscr{L}_N$ write $\phi=\phi(E_1,\ldots,E_n)$. Then \[ \widehat{E^k}(\phi) = N_k \, \phi(E_1,\ldots,E_{k-1},E_k^{\ast},E_{k+1},\ldots,E_n), \] where \[ E_k^{\ast} =(e_{k1},\ldots,e_{k,N_k-1}, -\sum_{j=1}^{N_k-1}e_{kj} -\sum_{\substack{i=1 \\ i \neq k}}^{n} \sum_{j=1}^{N_i} e_{ij}). \]
\end{lemma}

\begin{lemma}\label{hom-Shakiban}
Let $f=f(x_1,\ldots,x_n)$ be a homogeneous polynomial in $\mathbb{F}\{X\}$ with $\deg_{x_i}(f)>0$ for all $i$. If $E^{r}(f)=0$ for some $r\in\{1,\ldots,n\}$, then $E^{s}(f)=0$ for all $s\in\{1,\ldots,n\}$.
\end{lemma}

\begin{proof}
Without loss of generality, assume $r=1$ and $s=2$. Let $\mathscr{F}(f)=\phi=\phi(E_1,\ldots,E_n)$, where $E_k=(e_{k1},\ldots,e_{kN_k})$. By Lemma~\ref{lem:Shakiban},
$$
\widehat{E^1}(\phi)=N_1\,\phi(E_1^\ast,E_2,\ldots,E_n),
\qquad
\widehat{E^2}(\phi)=N_2\,\phi(E_1,E_2^\ast,E_3,\ldots,E_n),
$$
with
$$
E_1^\ast=\Bigl(e_{11},\ldots,e_{1,N_1-1},-\sum_{j=1}^{N_1-1}e_{1j}-\sum_{j=1}^{N_2}e_{2j}-\sum_{i=3}^{n}\sum_{j=1}^{N_i}e_{ij}\Bigr),
$$
$$
E_2^\ast=\Bigl(e_{21},\ldots,e_{2,N_2-1},-\sum_{j=1}^{N_1}e_{1j}-\sum_{j=1}^{N_2-1}e_{2j}-\sum_{i=3}^{n}\sum_{j=1}^{N_i}e_{ij}\Bigr).
$$
Since $\deg_{x_i}(f)>0$ for all $i$, we have $N_1,N_2\ge1$.

Set
$$
P=-\sum_{j=1}^{N_1-1}e_{1j}-\sum_{j=1}^{N_2}e_{2j}-\sum_{i=3}^{n}\sum_{j=1}^{N_i}e_{ij},
\qquad
Q=-\sum_{j=1}^{N_1-1}e_{1j}-\sum_{j=1}^{N_2-1}e_{2j}-\sum_{i=3}^{n}\sum_{j=1}^{N_i}e_{ij}.
$$
Then $P=Q-e_{2N_2}$, and we can rewrite
$$
\phi(E_1^\ast,E_2,\ldots,E_n)=\phi\!\bigl((e_{11},\ldots,e_{1,N_1-1},\,P),\ (e_{21},\ldots,e_{2,N_2-1},\,Q-P),\,E_3,\ldots,E_n\bigr),
$$
$$
\phi(E_1,E_2^\ast,E_3,\ldots,E_n)=\phi\!\bigl((e_{11},\ldots,e_{1,N_1-1},\,e_{1N_1}),\ (e_{21},\ldots,e_{2,N_2-1},\,Q-e_{1N_1}),\,E_3,\ldots,E_n\bigr).
$$

We note from the expressions above that the vector $e_{2N_2}$ occurs only in the linear combination defining $P$. Hence, by the definition of the vectors $e_{ij}$, the vectors $e_{11},\ldots,e_{1N_1-1}$ and $P$ are linearly independent. Furthermore,  the expressions $\phi(E_1^\ast,E_2,\ldots,E_n)$ and $\phi(E_1, E_2^\ast,E_3,\ldots,E_n)$ have the same forms with respect to their independent vectors.   Therefore, we conclude that  $\phi(E_1, E_2^\ast,E_3,\ldots,E_n)=0$ whenever $\phi(E_1^\ast,E_2,\ldots,E_n)=0$
and consequently $\widehat{E^1}(\phi)=0$ implies $\widehat{E^2}(\phi)=0$. This proves the claim.
\end{proof}

Next theorem shows that if a homogeneous polynomial in $Nov\langle X\rangle$ with $\deg_{x_i}(f)>0$ for all $i$, symmetric elements can be characterized in terms of Euler operators.

\begin{theorem}[symmetric element criterion]\label{criterion}
Suppose that $f=f(x_1,\ldots,x_n)$ is a homogeneous polynomial in $Nov\langle X\rangle$ with $\deg_{x_i}(f)>0$ for all $i$. Then the following are equivalent:
\begin{enumerate}
\item $f$ is a symmetric element with $\deg(f)>1$;
\item $f$ is a \emph{null Lagrangian}, i.e., $E^k(f)=0$ for all $k\in\{1,\ldots,n\}$;
\item $E^k(f)=0$ for some $k\in\{1,\ldots,n\}$.
\end{enumerate}
\end{theorem}

\begin{proof}
(1) $\Rightarrow$ (2): If $f$ is symmetric, then by Theorem~\ref{basis} there exists $g\in \mathbb{F}\{X\}$ with $f=D(g)$. By Theorem~\ref{Shakiban1981} (exactness of $0\to \mathbb{F}\{X\}\xrightarrow{D} \mathbb{F}\{X\}\xrightarrow{E} \mathbb{F}\{X\}$), we have $\ker E=\operatorname{im} D$, hence $E(f)=0$, i.e., $E^k(f)=0$ for all $k$.

(2) $\Rightarrow$ (3) is immediate.

(3) $\Rightarrow$ (1): Assume $f$ is homogeneous with $\deg_{x_i}(f)>0$ for all $i$ and $E^k(f)=0$ for some $k$. By Lemma~\ref{hom-Shakiban}, this implies $E^j(f)=0$ for all $j$. By Theorem~\ref{Shakiban1981}, there exists $g\in \mathbb{F}\{X\}$ with $f=D(g)$. Finally, Theorem~\ref{basis} implies that $f$ is symmetric.
\end{proof}

\begin{example} Let $f(x_1,x_2)=2x_1''x_1x_2+2x_1'^{2}x_2-3x_1^2x_2''-4x_1'x_1x_2'\in Nov(\{x_1,x_2\})$. Then $f(x_1,x_2)$ is homogeneous with $\deg_{x_1}(f)=2,\deg_{x_2}(f)=1$. Now we apply the Euler operators $E^1$ and $E^2$ to $f(x_1,x_2)$:
$$E^1(f)=2x_1''x_2+2(x_1x_2)''-4(x_1'x_2)'-6x_1x_2''-4x_1'x_2'+4(x_1x_2')'=0,$$
$$E^2(f)=2x_1''x_1+2x_1'^{2}-(3x_1^2)''+(4x_1'x_1)'=0.$$   
Note that $$f(x_1,x_2)=(2x_1'x_1x_2-3x_1^2x_2')'=\frac{5}{2}(x_1\circ x_1)\circ x_2-3(x_1\circ x_2)\circ x_1.$$
\end{example}

\section{A homomorphic image of $S\langle X\rangle$}
In this section, we show that the class of algebras embeddable into Novikov algebras with respect to the symmetrized product does not form a variety. We say that an algebra $B$ is embeddable into a Novikov algebra $A$ with respect to the symmetrized product if $B$ is isomorphic to a subalgebra of $A^{(+)}$, where $A^{(+)}$ denotes the symmetrization of $A$ with multiplication
$a\circ b=a\cdot b+b\cdot a$.

Assume that $X=\{x\}$. For convenience, put

$$x_i=D^i(x),\qquad i\geq 0.$$
\begin{prp}\label{non-speicality}
Let $\alpha$ be the ideal of $S\langle X\rangle$ generated by $D(x_0^3x_2)$. Then the quotient $S\langle X\rangle/\alpha$ cannot be embedded into the symmetrization of any Novikov algebra. 
\end{prp}

\begin{proof} We first show that 
$$S\langle X\rangle\cap \{\alpha\}
\not\subseteq \alpha$$
where $\{\alpha\}$ denotes the ideal of the free Novikov algebra $Nov\langle X\rangle$ generated by the set $\alpha$.

Let
$$
f=D(x_0^3x_2)
=x_0^3x_3+3x_0^2x_1x_2.
$$
Since
$$
\deg(x_0^3x_2)-d(x_0^3x_2)=4-2=2,
$$
this implies  that $f\in S\langle X\rangle$.
Equivalently, $f$ can be written explicitly in terms of the
symmetrized product as
$$
f=\frac12\left(
\bigl((x\circ x)\circ x\bigr)\circ x-
(x\circ x)\circ(x\circ x)
\right).
$$
Define
\begin{equation}\label{counter-element}
h=(f\cdot x)\cdot x+7x\cdot (f\cdot x)-7(x\cdot f)\cdot x-4x\cdot (x\cdot f)\in\{\alpha\}.
\end{equation}
Using the differential realization \(a\cdot b=a'b\), a direct calculation
gives
$$(f\cdot x)\cdot x=
x_0^5x_5
+10x_0^4x_1x_4
+12x_0^4x_2x_3
+24x_0^3x_1^2x_3
+21x_0^3x_1x_2^2
+12x_0^2x_1^3x_2,$$
$$
x\cdot(f\cdot x)=
x_0^4x_1x_4
+6x_0^3x_1^2x_3
+3x_0^3x_1x_2^2
+6x_0^2x_1^3x_2,
$$
$$
(x\cdot f)\cdot x=
x_0^4x_1x_4
+x_0^4x_2x_3
+6x_0^3x_1^2x_3+6x_0^3x_1x_2^2
+6x_0^2x_1^3x_2,$$
and
$$
x\cdot(x\cdot f)
=
x_0^3x_1^2x_3
+3x_0^2x_1^3x_2.$$
Substituting these expressions into \eqref{counter-element}, we obtain
\begin{equation}\label{counter-element-2}
 h=x_0^5x_5+10x_0^4x_1x_4+5x_0^4x_2x_3+20x_0^3x_1^2x_3.   
\end{equation}
On the other hand,
\begin{equation}\label{counter-element-3}h=
D\left(
x_0^5x_4+5x_0^4x_1x_3
\right).
\end{equation}
Indeed,
$$
D(x_0^5x_4)
=
x_0^5x_5+5x_0^4x_1x_4
\quad \mbox{and}\quad 
D(5x_0^4x_1x_3)
=
20x_0^3x_1^2x_3
+5x_0^4x_2x_3
+5x_0^4x_1x_4.
$$
Moreover,
$$\deg(x_0^5x_4)-d(x_0^5x_4)=6-4=2, \quad \mbox{and} \quad
\deg(x_0^4x_1x_3)-d(x_0^4x_1x_3)
=6-(1+3)=2.$$
Hence, by the basis theorem for symmetric elements, $h\in S\langle X\rangle$.
Thus $$h\in S\langle X\rangle\cap \{\alpha\}.$$
It remains to prove that
$$
h\notin\alpha.$$

The algebra $S\langle X\rangle$ is positively graded, and $f$ is homogeneous of
degree \(4\). Hence the degree-six component of the ideal generated by
$f$ is contained in the linear span of
$$
g_1=(f\circ x)\circ x \quad \mbox{and} \quad 
g_2=f\circ(x\circ x).$$

A direct calculation gives
\begin{equation}\label{expansion-1}
g_1=x_0^5x_5+12x_0^4x_1x_4+13x_0^4x_2x_3+37x_0^3x_1^2x_3+30x_0^3x_1x_2^2+27x_0^2x_1^3x_2,
\end{equation}
whereas
\begin{equation}\label{expansion-2}
g_2=2x_0^4x_1x_4+2x_0^4x_2x_3
+14x_0^3x_1^2x_3+12x_0^3x_1x_2^2
+18x_0^2x_1^3x_2.
\end{equation}

Suppose that $h\in\alpha$. Since $h$ has degree 6, there exist
scalars $\lambda,\mu$ such that
\[
h=\lambda g_1+\mu g_2.
\]
Comparing the coefficient of $x_0^5x_5$ in \eqref{counter-element-2}, \eqref{expansion-1}, and
\eqref{expansion-2}, we obtain $\lambda=1$. 
Comparing the coefficient of $x_0^4x_1x_4$, we get  $\mu=-1$.
However, with these values, the coefficient of $x_0^4x_2x_3$ in
$\lambda g_1+\mu g_2=g_1-g_2$ is 11,
whereas the corresponding coefficient in $h$ is $5$. This is a
contradiction. Hence
$$h\notin\alpha.$$
We have therefore constructed an element satisfying
$$
h\in S\langle X\rangle\cap \{\alpha\}
\qquad\text{and}\qquad
h\notin\alpha.
$$
Consequently,
$$
S\langle X\rangle\cap \{\alpha\}
\not\subseteq\alpha.$$ So, by Cohn’s criterion \cite[Theorem 2.2]{Cohn1954} the factor-algebra $
S\langle X\rangle/\alpha$ can not be embedded into the free  Novikov algebra $Nov\langle X\rangle$ with respect to the symmetrized product.

\end{proof}

\begin{corollary}\label{not variety}
The class of algebras embeddable into  Novikov algebras with respect to symmetrized product $a\circ b=a\cdot b+b\cdot a$ does not form a variety.
\end{corollary}

\begin{proof}
Let $\mathcal{V}$ be the class of all algebras embeddable into symmetrizations of Novikov algebras. By Proposition \ref{non-speicality}, the class $\mathcal{V}$ is not closed under homomorphic images. Therefore, $\mathcal{V}$ does not form a variety of algebras.
\end{proof}

\section{\label{2}\ $S_n$-module structures}

In this section, we study the $S_n$-module structure of the space $Sym_{n}$. We refer the reader to \cite{Dzhumadil'daevIsmailov2014, Sagan} for the required definitions and notations.

We consider $F_{n}$ as an $S_{n}$-module with the natural action
\[
\sigma\cdot X(x_1,\ldots,x_{n})
= X\bigl(x_{\sigma(1)},\ldots,x_{\sigma(n)}\bigr).
\]
Let $\alpha$ be a partition of $n$, and let $P(n)$ denote the set of partitions of $n$.
For $\alpha=(n^{i_n},\ldots,2^{i_2},1^{i_1})\in P(n)$, define the partition $w(\alpha)\in P(n+2)$ by
\[
w(\alpha)=sort\!\left(n+2-\sum_{j=1}^{n} i_j,\; i_1,\; i_2,\; \ldots,\; i_n\right).
\]
Here $sort()$ means reordering the entries into a non-increasing sequence. It is easy to see that $Sym_1$ and $Sym_2$ are isomorphic to the trivial $S_n$-modules. The following theorem completely describes the $S_{n+2}$-module structure of $Sym_{n+2}$ for $n\geq 1$.

\begin{theorem}
Let $Sym_{n+2}$ be the multilinear component of degree $n+2$ of $S\langle X\rangle$. 

{\bf a.} As an $S_{n+2}$-module, $Sym_{n+2}$ is isomorphic to a direct sum of modules induced from the trivial module of the Young subgroups $S_{w(\alpha)}$:
$$Sym_{n+2}\cong \bigoplus_{\alpha\vdash n} Ind_{S_{w(\alpha)}}^{S_{n+2}}(\bf1).$$

{\bf b.} Moreover, each induced module decomposes into a direct sum of irreducible modules, and $$Sym_{n+2}\cong\bigoplus_{\beta \vdash n+2}(\sum_{\alpha\vdash n} K_{\beta w(\alpha)})S^{\beta},$$
where $S^{\beta}$ is the Specht module corresponding to the partition $\beta$, and
$K_{\beta\,w(\alpha)}$ is the Kostka number counting semistandard Young tableaux of shape $\beta$
and content $w(\alpha)$.

{\bf c.} An irreducible $S_{n+2}$-module is called {\it admissible} if it appears in the decomposition of $Sym_{n+2}$. Then $\beta=(\beta_{1},\ldots,\beta_{k})\vdash n+2$ is admissible if and only if $$\beta_{1}-2\geq\beta_{3}+2\beta_{4}+\ldots+(k-2)\beta_{k}.$$

\end{theorem}

\begin{proof} Let $\alpha=(\alpha_1,\ldots,\alpha_n)\vdash n$ be a partition of $n$, and let $F_{\alpha}$ be the subspace of $Sym_{n+2}$ spanned by the basis elements of the form
\[
\bigl(x_{j_1}^{(\alpha_1)}\cdots x_{j_n}^{(\alpha_n)}x_{j_{n+1}}x_{j_{n+2}}\bigr)'.
\]
Then $F_{\alpha}$ is an $S_{n+2}$-submodule of $Sym_{n+2}$ for every $\alpha\vdash n$.
Write $\alpha$ in the form $\alpha=(n^{i_n},\ldots,2^{i_2},1^{i_1})$.
Then the number of variables occurring in elements of $F_{\alpha}$ with differential degree $0$ equals
\[
n+2-\sum_{j=1}^{n} i_j.
\]
Furthermore, the symmetric group acts trivially on the variables of each fixed differential degree $j$.
Therefore, upon restricting the action, $F_{\alpha}$ is isomorphic (as an $S_{n+2}$-module) to the permutation module
$M^{w(\alpha)}$ corresponding to $w(\alpha)$ (see \cite{Dzhumadil'daevIsmailov2014} or \cite[Def.~2.1.5]{Sagan}).

Recall that, in general, the permutation module $M^\alpha$ is isomorphic to the induced module $Ind_{S_\alpha}^{S_n}(\mathbf{1})$, where $\mathbf{1}$ denotes the trivial representation of the Young subgroup $S_\alpha\subset S_n$. It follows that
\[
F_\alpha \cong Ind_{S_{w(\alpha)}}^{S_{n+2}}(\mathbf{1}).
\]
Therefore, we obtain the following isomorphism of $S_{n+2}$-modules:
\[
Sym_{n+2} \cong \bigoplus_{\alpha \vdash n} F_\alpha
\cong \bigoplus_{\alpha \vdash n} Ind_{S_{w(\alpha)}}^{S_{n+2}}(\mathbf{1}).
\]
By Young's rule,
\[
Ind_{S_{w(\alpha)}}^{S_{n+2}}(\mathbf{1})
=\bigoplus_{\beta \unrhd w(\alpha)} K_{\beta\, w(\alpha)}\, S^\beta .
\]
Consequently,
\[
Sym_{n+2} \cong \bigoplus_{\alpha \vdash n} \ \bigoplus_{\beta \unrhd w(\alpha)}
K_{\beta\, w(\alpha)}\, S^\beta
=
\bigoplus_{\beta \vdash n+2}\left(\sum_{\alpha \vdash n} K_{\beta\, w(\alpha)}\right) S^\beta .
\]

To prove part {\bf (c)}, we repeat the argument from Section~8 of \cite{Dzhumadil'daevIsmailov2014} for $\beta\vdash n+2$.

\end{proof}

\section*{Acknowledgments}

The authors are supported by grant AP26199089 from the Ministry of Science and Higher Education of the Republic of Kazakhstan.  We thank Vladimir Dotsenko for his kind interest in our results and for drawing our attention to Molev’s result \cite{Molev1989}. The second author is grateful to Aibol Orazgaliyev for valuable discussions.

\end{document}